\documentclass[11pt]{article}
\usepackage{amssymb}
\usepackage[all,arc]{xy}
\usepackage{mathrsfs}
\usepackage{extarrows}
\usepackage{xcolor}
\usepackage{amsmath}
\usepackage{amsfonts,epsfig}
\usepackage{amssymb,bm}
\usepackage{graphicx}
\usepackage[T1]{fontenc}
\usepackage{subfig}
\usepackage{pgf}
\usepackage{appendix}

\usepackage{tabularx}  
\usepackage{booktabs}
\usepackage{multirow}
\usepackage{array,makecell}

\usepackage{algorithm}
\usepackage{algpseudocode}
\usepackage{pifont}

\usepackage{tikz}
\usetikzlibrary{decorations.pathreplacing}

\numberwithin{equation}{section}

\usepackage{amsmath,amsthm,amsfonts,amssymb,amscd,color}
\usepackage{hyperref}
\hypersetup{
	colorlinks=true,
	linkcolor=blue,
	anchorcolor=blue,
	citecolor=blue}

\allowdisplaybreaks[4]

\newtheorem{theo}{Theorem}[section]
\newtheorem{lem}{Lemma}[section]

\newtheorem{defi}{Definition}[section]

\newtheorem{que}{Question}[section]

\newcommand{\Z}{\mathbb{Z}}
\newcommand{\F}{\mathbb{F}}

\newcommand{\co}{\mathrm{co}}

\newcommand{\vp}{v_p}

\newcommand{\p}{\mathfrak{p}}

\usepackage{amsthm,amssymb}

\newcommand{\ex}{{\rm ex}}

\begin{document}

	\title{On linear $k$-graphs with codegree Tur\'an density arbitrarily close to zero
		 }
	\author{Xiaona Fang, Yaojun Chen\thanks{Corresponding author: yaojunc@nju.edu.cn}\\
		{\small School of Mathematics, Nanjing University,} {\small Nanjing, 210093, P.R. China}
}
	\date{}
	\maketitle
	
	\begin{abstract}	
		Let $F$ be a $k$-uniform hypergraph, abbreviated as $k$-graph. The codegree Tur\'an density $\pi_{\mathrm{co}}(F)$ is the supremum over all $\gamma \in [0,1)$ such that, for arbitrarily large $n$, there exists an $n$-vertex $F$-free $k$-graph $H$ whose every $(k-1)$-subset of vertices lies in at least $\gamma n$ edges. In this paper, we prove that there is a linear $k$-graph $F$ with $0<\pi_{co}(F) < \varepsilon$ for any $\varepsilon>0$. The special case $k=3$ solve a question proposed by Ding, Lamaison, Liu, Wang and Yang (JLMS, 2025). The main method combines an affine-plane-type incidence structure over a finite field and elementary number-theoretic arguments.

		
	\end{abstract}
	
	\section{Introduction}\label{sec1}

Given a $k$-uniform hypergraph $F$ (or simply a $k$-graph), the Tur\'an number of $F$, denoted by $\ex(n,F)$, is the maximum number of edges in an $n$-vertex $k$-graph that contains no copy of $F$. Tur\'an-type problems lie at the heart of extremal combinatorics and trace back to the pioneering works of Mantel and Tur\'an in the early twentieth century.
To study the asymptotic behavior of $\ex(n,F)$, one introduces the Tur\'an density
\[
\pi(F):=\lim_{n\to\infty}\frac{\ex(n,F)}{\binom{n}{k}}.
\]
For ordinary graphs (that is, when $k=2$), Tur\'an densities are relatively well understood. In contrast, the situation for hypergraphs with $k\ge 3$ is considerably more challenging. Despite decades of intensive study, even the Tur\'an densities of the two 3-graphs on four vertices with three and four edges, namely $K_4^{(3)-}$ and $K_4^{(3)}$, remain unknown.

A natural variant of Tur\'an density, introduced by Mubayi and Zhao~\cite{mubayi}, is the codegree Tur\'an density. For a $k$-graph $H$ and a vertex set $S\subseteq V(H)$, let $d_H(S)$ denote the number of edges containing $S$. The minimum codegree of $H$, denoted by $\delta_{\mathrm{co}}(H)$, is defined as the minimum of $d_H(S)$ taken over all $(k-1)$-subsets $S$ of $V(H)$. The codegree Tur\'an number $\ex_{\mathrm{co}}(n,F)$ is the largest possible value of $\delta_{\mathrm{co}}(H)$ among all $n$-vertex $F$-free $k$-graphs $H$, and the corresponding codegree Tur\'an density is
\[
\pi_{\mathrm{co}}(F):=\lim_{n\to\infty}\frac{\ex_{\mathrm{co}}(n,F)}{n}.
\]
It is known that this limit always exists~\cite{mubayi}, and it is not hard to see that
$\pi_{\mathrm{co}}(F)\le \pi(F)$.

 Mubayi and Zhao \cite{mubayi} showed that $\{\pi_{\co}(\mathcal{F}): \mathcal{F} \text{ is a family of } k\text{-graphs}\}$ is dense in $[0,1]$.
Gao, Pikhurko, Rong and Sun \cite{gao} proved that for every rational number $\alpha \in [0,1)$, there exists a finite family of $k$-graphs $\mathcal{F}$ such that $\pi_{\co}(\mathcal{F})=\alpha$.
For the codegree Tur\'an density of a $k$-graph, it has been proved that $0$ and $1-\frac{1}{r}$ for any $r\in \mathbb{N}$ are the accumulation points of
$\{\pi_{\co}(F): F \text{ is a } k\text{-graph}\}$ \cite{li, piga}. A hypergraph is said to be linear if every pair of distinct hyperedges intersects in at most one vertex.
Ding, Lamaison, Liu, Wang and Yang \cite{ding} study the problem of what 3-graphs $F$ satisfy $\pi_{\co}(F) = 0$. They proposed a conjecture and  reduced the problem to the linear $3$-graph case. The codegree Tur\'an density can be arbitrarily close to zero, which partly explains the difficulty of the problem. So they asked the following question.

\begin{que}(Ding, Lamaison, Liu, Wang and Yang, \cite{ding})\label{que}
	For any $\varepsilon>0$, is there a linear $3$-graph $F$ with $0<\pi_{\co}(F) < \varepsilon$?
\end{que}

In this paper, we consider a more general question: is there  a linear $k$-graph $F$ with $0<\pi_{co}(F) < \varepsilon$ for any $\varepsilon>0$, which includes Question \ref{que} ($k=3$) as a special case.  
 Before presenting our conclusion, we first define a linear $k$-graph as follows. The construction is based on an affine-plane-type incidence structure over a finite field.

Let $\Z_n=\{0,1,2,\dots,n-1\}$ denote the additive group of integers
modulo $n$, and $\mathbb{F}_{\p}$ denote a finite field of $\p$ elements.
\begin{defi}\label{O^k}
	For integers $\ell \geq k\geq 3$, let $\p$ be a prime with $\p \geq k-1$ and $\p\neq k$, let $\alpha_0=0\in \mathbb{F}_{\p}$ and $\{ \alpha_1, \dots, \alpha_{k-2} \}\subset \F_{\p}$ be a set of $k-2$ distinct nonzero slopes, we define
	the $k$-uniform linear hypergraph $O^{(k)}_{\ell,\p}$of length $\ell$ as follows. The vertex set is \[V(O^{(k)}_{\ell,\p}) = \{ v_{i,j} ^{t}: i\in \Z_{\p}, j\in \Z_{k-1}, t\in \Z_{\ell} \}.\]
Let
	\[
	e(i,j,t)=\{v^{t+1}_{0,j}\}\cup
	\{v^t_{\alpha_jx+i,x}:x\in\mathbb Z_{k-1}\},
	\]
    where the sum at different superscripts and subscripts is performed in the corresponding additive group.
	The edge set of $O^{(k)}_{\ell,\p}$ is
	\[
	E(O^{(k)}_{\ell,\p})=
	\{e(i,j,t):i\in \Z_{\p}, j\in \Z_{k-1}, t\in \Z_{\ell} \}.
	\]
\end{defi}

We will show that $O^{(k)}_{\ell,\p}$ is linear in Section \ref{sec3}.
As an example, the linear 3-graph $O_{6,2}^{(3)}$ is illustrated in Figure \ref{Ok}.

\begin{figure}
	\centering
	\begin{tikzpicture}
		
		\newcommand{\MyShape}{
			\filldraw[
			fill=red!8,
			fill opacity=0.35,
			draw=red
			]
			({1.5*cos(-60)}, {1.5*sin(-60)})
			to[bend left=20]
			({1*cos(-120)}, {1*sin(-120)})
			to[bend left=20]
			({2.5*cos(-120)}, {2.5*sin(-120)})
			to[bend left=10]
			({1.5*cos(-60)}, {1.5*sin(-60)})
			-- cycle;
			
			\filldraw[
			fill=red!8,
			fill opacity=0.35,
			draw=red
			]
			({1.5*cos(-60)}, {1.5*sin(-60)})
			to[bend left=8]
			({1.5*cos(-120)}, {1.5*sin(-120)})
			to[bend left=8]
			({2*cos(-120)}, {2*sin(-120)})
			to[bend left=4]
			({1.5*cos(-60)}, {1.5*sin(-60)})
			-- cycle;
			
			\filldraw[
			fill=blue!8,
			fill opacity=0.35,
			draw=blue
			]
			({cos(-60)}, {sin(-60)})
			to[bend left=8]
			({cos(-120)}, {sin(-120)})
			to[bend left=8]
			({1.5*cos(-120)}, {1.5*sin(-120)})
			to[bend left=4]
			({cos(-60)}, {sin(-60)})
			-- cycle;
			
			\filldraw[
			fill=blue!8,
			fill opacity=0.35,
			draw=blue
			]
			({cos(-60)}, {sin(-60)})
			to[bend left=1]
			({2*cos(-120)}, {2*sin(-120)})
			to[bend left=8]
			({2.5*cos(-120)}, {2.5*sin(-120)})
			to[bend left=1]
			({cos(-60)}, {sin(-60)})
			-- cycle;
			
			\filldraw ({cos(-60)}, {sin(-60)}) circle (0.03);
			\filldraw ({1.5*cos(-60)}, {1.5*sin(-60)}) circle (0.03);
			\filldraw ({2*cos(-60)}, {2*sin(-60)}) circle (0.03);
			\filldraw ({2.5*cos(-60)}, {2.5*sin(-60)}) circle (0.03);
		}
		
		\MyShape
		
		\begin{scope}[rotate around={-60:(0,0)}]
			\MyShape
		\end{scope}
		\begin{scope}[rotate around={-120:(0,0)}]
			\MyShape
		\end{scope}
		\begin{scope}[rotate around={-180:(0,0)}]
			\MyShape
		\end{scope}
		\begin{scope}[rotate around={-240:(0,0)}]
			\MyShape
		\end{scope}
		\begin{scope}[rotate around={-300:(0,0)}]
			\MyShape
		\end{scope}	
	
	\node at ({cos(-120)+0.15}, {sin(-120)+0.25}) {\tiny $v_{0,0}^1$};
	\node at ({cos(-60)-0.15}, {sin(-60)+0.25}) {\tiny $v_{0,0}^2$};
	
	\node at ({1.5*cos(-120)-0.3}, {1.5*sin(-120)}) {\tiny $v_{0,1}^1$};
	\node at ({2*cos(-120)-0.3}, {2*sin(-120)}) {\tiny $v_{1,0}^1$};
	\node at ({2.5*cos(-120)-0.3}, {2.5*sin(-120)}) {\tiny $v_{1,1}^1$};
	
	\node at ({1.5*cos(-60)-0.1}, {1.5*sin(-60)-0.22}) {\tiny $v_{0,1}^2$};
	\node at ({2*cos(-60)-0.1}, {2*sin(-60)-0.22}) {\tiny $v_{1,0}^2$};
	\node at ({2.5*cos(-60)-0.1}, {2.5*sin(-60)-0.22}) {\tiny $v_{1,1}^2$};
	\end{tikzpicture}
	\caption{$O_{6,2}^{(3)}$}
	\label{Ok}
\end{figure}

\newpage
The main result of this paper is as follows.

\begin{theo}\label{th1.1}
		For any $\varepsilon>0$, there is an $\ell\in \mathbb{N}$ such that $0< \pi_{\co} (O^{(k)}_{\ell,\p} ) < \varepsilon$.
\end{theo}
Clearly, Theorem \ref{th1.1} answers Question \ref{que} in affirmative.
\section{Preliminaries}

In this section, we introduce some additional definitions and some known results for the purpose to prove Theorem \ref{th1.1}.

\begin{defi}	
	Let $m\ge 2$ be an integer. We define the $n$-vertex $k$-graph $H_m^{(k)}(n)$ as follows. Partition the vertex set of $H_m^{(k)}(n)$ into $m$ almost equal parts, i.e., $V(H_m^{(k)}(n) )= V_0\cup V_1\cup\cdots \cup V_{m-1}$ with $|V_i|=\lfloor n/m\rfloor$ or $\lceil n/m\rceil$.
	For a vertex $u$, we write
	$\chi(u)=s \in \Z_m$
	if and only if
	$ u\in V_s$. And $\{u_1,\dots,u_k\}$ is an edge of $H_m^{(k)}(n)$ if and only if
	\[
	\chi(u_1)+\cdots+\chi(u_k)\equiv 1\pmod m.
	\]	
\end{defi}

\begin{defi}
	Let $c,n$ be positive integers and let $F$ be a $k$-graph on $[n]$. 
	The $c$-blow-up of $F$, denoted by $F(c)$, is the $n$-partite $k$-graph $(V,E)$ with
	$
	V = V_1 \cup V_2 \cup \cdots \cup V_n,
	$
	every $|V_i|=c$ and
	$
	E=\bigl\{\{v_{i_1},v_{i_2},\ldots,v_{i_k}\}: 
	v_{i_j}\in V_{i_j},\ \{i_1,i_2,\ldots,i_k\}\in E(F)\bigr\}.
	$
\end{defi}

\begin{lem}(Mubayi and Zhao, \cite{mubayi})\label{lem_blow}
	 For every positive integer $c$, $\pi_{\co}(F) = \pi_{\co}(F(c ))$.
\end{lem}

\begin{defi}
	For integers $\ell \geq k \geq 3$, the $k$-uniform zycle of length $\ell$ is the $k$-graph $Z_{\ell}^{(k)}$ given by
	\[
	V\bigl(Z_{\ell}^{(k)}\bigr)
	=
	\{v_i^j : i\in \Z_{\ell},\, j\in \Z_{k-1} \}, \text{ and }
	\]
	\[
	E\bigl(Z_{\ell}^{(k)}\bigr)
	=
	\left\{
	v_i^0 v_i^1 \cdots v_i^{k-2}v_{i+1}^j
	:
	i\in \Z_{\ell},\, j\in \Z_{k-1}
	\right\},
	\]
\end{defi}

\begin{lem}(Piga and Sch\"ulke, \cite{piga})\label{lem_Z}
	For any $d \in (0,1]$, there is an $\ell\in \mathbb{N}$ such that $\pi_{\co} (Z_{\ell}^{(k)} ) \leq d$.
\end{lem}

For a prime $p$ and a nonzero integer $N$, let
	\[
	\vp(N)=\max\{a\ge 0:p^a\mid N\}
	\]
	be the $p$-adic valuation of $N$.
    
The following is the Lifting-the-Exponent Lemma, whose roots trace back to the closely related results of Birkhoff and Vandiver \cite{BirkhoffVandiver1904}.

\begin{lem} (Lifting-the-Exponent Lemma, \cite{Andreescu})\label{lem}
	For any integers $x$ and $y$, a positive integer $n$, and a prime number $p$ such that $p\nmid x$ and $p\nmid y$, the following statements hold:
	
	(i) if $p$ is odd and $p\mid x-y$, then $\vp(x^{n}-y^{n})=\vp(x-y)+\vp(n)$;
	
	(ii) if $p=2$ and $4\mid x-y$, then $v_2(x^{n}-y^{n})=v_2(x-y)+v_2(n)$.
\end{lem}

\section{Proof of the main theorem}\label{sec3}

\begin{lem}
	The $k$-uniform hypergraph $O^{(k)}_{\ell,\p}$ is linear.
\end{lem}
\begin{proof}	
	Take two distinct edges $e(i,j,t)$ and $e(i',j',t')$. If their layer-sets
	$\{t,t+1\}$ and $\{t',t'+1\}$ are disjoint, then the two edges are disjoint.
	If $t'=t+1$ or $t'=t-1$, then they contains at most one common vertex at their common layer.
	
	It remains to consider the case $t'=t$. If $j=j'$, then their vertices in layer $t$ can coincide only for some
	$x\in\mathbb Z_{k-1}$ satisfying
	$\alpha_jx+i=\alpha_{j}x+i'$ in $\Z_{\p}$. Hence, $i=i'$. It is a contradiction.
	
	If $j\neq j'$, then $\alpha_j\neq\alpha_{j'}$, i.e., $\alpha_j-\alpha_{j'}$ is invertible in $\F_{\p}$, and hence
	$
	(\alpha_j-\alpha_{j'})x=i'-i
	$
	has at most one solution in $\mathbb F_{\p}$. Thus it has at most one solution $x$ in $\mathbb \Z_{k-1}$. 
	
	Therefore, two distinct edges have
	at most one common vertex. 
\end{proof}

\noindent\textbf{Proof of Theorem \ref{th1.1}}. 
	We first prove the upper bound. Since $\{v^t_{i,j}: i\in \Z_{\p} \}$ is the blow-up vertex set of the vertex $v^t_{0,j}$ for any $t\in\Z_{\ell} $ and $j\in \Z_{k-1}$, $O^{(k)}_{\ell,\p}$ is a subhypergraph of the $\p$-blow-up of $Z^{(k)}_{\ell}$. Thus by Lemmas \ref{lem_blow} and \ref{lem_Z}, for any $\varepsilon>0$,
	there exist $0<d<\varepsilon$ and an $\ell\in \mathbb{N}$ such that $ \pi_{\co}(O^{(k)}_{\ell,\p}) \leq \pi_{\co} (Z_{\ell}^{(k)} ) \leq d <\varepsilon$.
	
	
	Now we prove the lower bound.	
	We choose a prime $p$ dividing $k$ as follows.
	
	(i) If $k$ has an odd prime divisor, choose one such divisor and call it $p$.
	
	(ii) If $k$ has no odd prime divisor, then $k$ is a power of $2$ and $4 \mid k$ since $k \geq  3$. In
	this case, choose $p = 2$.
	
	Set
	\[
	q=\vp(\ell)+1
	\qquad\text{and}\qquad
	m=p^q.
	\]   
For every $(k-1)$-set $S=\{u_1,\dots,u_{k-1}\}$ in $H_m^{(k)}(n)$, a vertex $y\in N(S)$ if and only if $\chi(y)=j$, where
\[
j\equiv 1-\sum_{i=1}^{k-1}\chi(u_i)\pmod m.
\]
Thus $ V_j \setminus S \subseteq N(S)$.  Therefore
\begin{equation}\label{eq:codegree-lower-bound}
	\delta_{\co}(H_m^{(k)}(n))
	\ge
	\left\lfloor\frac nm\right\rfloor-(k-1).
\end{equation}

	We will prove that $H_m^{(k)}(n)$ is $O^{(k)}_{\ell,\p}$-free.
	Assume, for contradiction, that $H_m^{(k)}(n)$ contains a copy of
	$O^{(k)}_{\ell,\p}$. By Definition \ref{O^k}, for $t\in \Z_{\ell}$ and $j_0 \in\Z_{k-1}$, any vertex $v_{0,j_0}^{t +1}$ is contained in $\p$ hyperedges which cover each vertex in the $t$ layer $\{ v_{i,j} ^{t}: i\in \Z_{\p}, j\in \Z_{k-1}\}$ exactly once. Thus, 
    $$\p \cdot \chi(v_{0,j_0}^{t +1}) +  \sum_{i\in \Z_{\p}, j\in \Z_{k-1}} \chi( v_{i,j} ^{t} ) \equiv \p \pmod m.$$
	Since $p\mid k$ and $\gcd (\p,k)=1$, we have $\gcd(\p,m)=\gcd(\p,p^q) =1$. Therefore, 
	\[
	\chi(v_{0,0}^{t +1})=\chi(v_{0,1}^{t +1})=\chi(v_{0,2}^{t +1})=\cdots=\chi(v_{0,k-2}^{t +1})
	~\text{in~}\Z_m.
	\]
	Let $x_t=\chi(v_{0,0}^{t})=\chi(v_{0,1}^{t})=\cdots=\chi(v_{0,k-2}^{t}) \in\Z_m.$ Since $ v_{0,0}^{t +1} v_{0,0}^{t} v_{0,1}^{t} \cdots v_{0,k-2}^{t} $ is an hyperedge in $O^{(k)}_{\ell,\p}$, we have 
	\begin{equation}\label{eq:recurrence}
		x_{t+1}\equiv 1-(k-1)x_t\pmod m.
	\end{equation}
	 Define $Y_t=kx_t-1\in\Z_{km}$.
	From \eqref{eq:recurrence}, multiplying by $k$ gives a congruence modulo $km$:
	$kx_{t+1}-1
	\equiv
	k(1-(k-1)x_t)-1 =(1-k)(kx_t-1)
	\pmod {km}.$ That is 
	\begin{equation}\label{eq:Yrecurrence}
		Y_{t+1}\equiv (1-k)Y_t\pmod {km}.
	\end{equation}
	Iterating \eqref{eq:Yrecurrence} for $\ell$ steps gives
	\begin{equation*}\label{eq:Yiterate}
	Y_0 \equiv	Y_{\ell}\equiv (1-k)^{\ell}Y_0\pmod {km}.
	\end{equation*}	
	We claim that $\gcd(Y_0,km)=1$. In fact, every prime divisor $\rho$ of $km$ divides $k$, since $m=p^q$ and $p\mid k$. 
	Hence $Y_0=kx_0-1\equiv -1\pmod \rho.$ Therefore,
	$Y_0$ is invertible modulo $km$, and we have
	\begin{equation}\label{eq:main-congruence}
		(1-k)^{\ell}\equiv 1\pmod {km}.
	\end{equation}
	Since $m=p^{\vp(\ell)+1} $, the integer $km$
	is divisible by
	$p^{\vp(k)+\vp(\ell)+1}$.
	Thus \eqref{eq:main-congruence} implies
	\begin{equation}\label{eq:valuation-lower}
		\vp((1-k)^{\ell}-1)
		\ge
		\vp(k)+\vp(\ell)+1.
	\end{equation}
	By Lemma \ref{lem} and the choice of $p$, we have	
	\begin{equation}\label{eq:valuation-exact}
		\vp((1-k)^{\ell}-1)=\vp(k)+\vp(\ell).
	\end{equation}
	Equations \eqref{eq:valuation-lower} and \eqref{eq:valuation-exact} contradict
	each other.  This contradiction proves that $H_m^{(k)}(n)$ is
	$O^{(k)}_{\ell,\p}$-free.
	Therefore, 
	\[
	\pi_{\co}(O^{(k)}_{\ell,\p})
	\ge
	\liminf_{n\to\infty}
	\frac{\left\lfloor n/m\right\rfloor-(k-1)}{n}
	=
	\frac1m =
	\frac{1}{p^{\vp(\ell)+1}}
	>0.
	\]

	The proof is completed.$\hfill \square$

\vskip 5mm
\section*{Acknowledgments}
Many thanks to Lanchao Wang for his helpful suggestions in preparing this version of the paper. This research was supported by National Key R\&D Program of China under grant number 2024YFA1013900 and NSFC under grant number 12471327.

\vskip 5mm
\noindent\textbf{Data availability statement} \  Data sharing not applicable to this article as no datasets were generated or analyzed during the current study.

\vskip 5mm
\noindent\textbf{Declarations of conflict of interest}\  The authors declare that they have no known competing financial interests or personal relationships that could have appeared to influence the work reported in this paper.

\end{document}